\documentclass[12pt]{article}
\usepackage[a4paper]{anysize}\marginsize{2.5cm}{2.5cm}{1.5cm}{2cm}
\pdfpagewidth=\paperwidth \pdfpageheight=\paperheight
\usepackage{pgf,tikz,hyperref}
\usetikzlibrary{arrows}
\usepackage{amsfonts,amssymb,amsthm,amsmath,eucal,extarrows}
\usepackage{graphicx}

\theoremstyle{plain}
\newtheorem{thm}{Theorem}
\newtheorem{theorem}[thm]{Theorem}

\newtheorem{proposition}[thm]{Proposition}

\newtheorem{conjecture}[thm]{Conjecture}

\theoremstyle{definition}
\newtheorem{definition}[thm]{Definition}
\newtheorem{remark}[thm]{Remark}
\newtheorem{example}[thm]{Example}

\newtheorem{question}[thm]{Question}

\newtheorem{thevarthm}[thm]{\varthmname}

\newenvironment{varthm*}[1]{\trivlist\item[]{\bf #1.}\it}{\endtrivlist}

\newcommand\be{\begin{eqnarray*}}
\newcommand\ee{\end{eqnarray*}}

\renewcommand\P{\mathbb P}

\newcommand\newop[2]{\def#1{\mathop{\rm #2}\nolimits}}
\newop\edim{edim}
\newop\Zeroes{Zeroes}
\newop\Jac{Jac}
\newop\Ass{Ass}
\newop\SL{SL}
\newop\PGL{{\P}GL}
\newop\Km{Km}

\newcommand\keywords[1]{{\renewcommand\thefootnote{}\footnotetext{\textit{Keywords:} #1.}}}
\newcommand\subclass[1]{{\renewcommand\thefootnote{}\footnotetext{\textit{Mathematics Subject Classification (2010):} #1.}}}

\begin{document}

\author{Piotr Pokora and Tomasz Szemberg}
\title{Conic-line arrangements in the complex projective plane}
\date{\today}
\maketitle
\thispagestyle{empty}

\begin{abstract}
The main goal of this note is to begin a systematic study on conic-line arrangements in the complex projective plane. We show a de Bruijn-Erd\H{o}s-type inequality and Hirzebruch-type inequality for a certain class of conic-line arrangements having ordinary singularities. We will also study, in detail, certain conic-line arrangements in the context of the geography of log-surfaces and free divisors in the sense of Saito.
\end{abstract}

\subclass{14C20, 14N20, 32S22}


\section{Introduction}

The main goal of the present work is to begin a systematic studies on the theory of conic-line arrangements in the complex projective plane. The theory of line arrangements is a classical and rich subject of studies with many deep results with applications and impact in numerous branches of mathematics. For example the Hirzebruch inequality \cite{BHH87,Hirzebruch, Hirzebruch1} appreciated very much in combinatorics is motivated by problems in algebraic geometry and is derived with its methods.
It is the most attractive feature of line arrangements theory that sits on the boundary of combinatorics, commutative algebra, topology, algebraic geometry, etc. However, if we look at the theory of curve arrangements in the plane, to our surprise, a lot of work has to be done in order to reach the same level of understanding. As a starting point for our discussion, recall that for configurations of points and conics in the plane we have an analogue of the Sylvester-Gallai theorem on ordinary lines, namely if $S$ is a finite set of points in the real projective plane and no conic contains all the points of $S$, then there is a conic which contains exactly five points of $S$. This result was proved by Wiseman and Wilson in \cite{Wis}, and it was recently reproved by two distinct research groups in \cite{Boys, Czap}. It is worth mentioning that the proof presented in \cite{Czap} is based on a certain reduction that uses Hirzebruch's inequality for dual configurations of points and $r$-rich lines. From the combinatorial perspective, it is natural to generalize results from the theory of point-line configurations to arrangements of smooth plane curves and this is one of the main goals of the present paper. More precisely, our aim is to develop new techniques that allows to understand the combinatorics and geometry of rational curve arrangements. It seems very natural to start working on a smaller subclass of such arrangements, namely arrangements of smooth rational curves. We refer to \cite{IgorD} for a nice introduction to abstract arrangements which contains a theoretical background for our work.

Now we are making our first assumption, namely we are going to consider arrangements of smooth rational curves (smooth conics and lines) having only ordinary singularities. This assumption might be considered as a strong restriction, but the most important advantage of this approach is that we can apply some combinatorial methods which are hardly applicable if we start to work in the whole generality.

If $\mathcal{CL} = \{\ell_{1}, ...,\ell_{d}, C_{1}, ..., C_{k}\} \subset \mathbb{P}^{2}_{\mathbb{C}}$ is an arrangement of $d$ lines and $k$ conics having only ordinary singularities, i.e., the intersection points look locally as $\{x^{k}=y^{k}\}$ for some integer $k\geq 2$, then we have the following combinatorial count (by B\'ezout)
$$ 4 \binom{k}{2} + \binom{d}{2} +2kd = \sum_{r\geq 2}\binom{r}{2}t_{r},$$
where $t_{r}$ denotes the number of $r$-fold points, i.e.,  points where exactly $r$ curves from $\mathcal{CL}$ meet. We will use also the following abbreviations:
$$f_{0} := \sum_{r\geq 2}t_{r}, \quad f_{1}:=\sum_{r\geq 2}rt_{r}.$$

The most important part in the local study of such arrangements is the number of analytic branches $r_{p}$ of a singular point $p$, but since all intersection points are ordinary, this number coincides with the multiplicity of a given singular point, and we are going to use ${\rm mult}_{p}$ (or shortly $m_{p}$) in order to denote the multiplicity, and abusing the notation we will denote also by $m_{p}$ the number of branches.

Now we are going to give an outline of the results obtained in the paper.

 Section \ref{sec: Ch arrangement} can be viewed as a warm-up, we are going to present an interesting construction of an arrangement of $12$ conics with $9$ points of multiplicity $8$, and this construction is strictly related to the Hesse arrangement of $12$ lines with $9$ points of multiplicity $4$. Moreover, it turns out that this construction leads to a new arrangement of conics and lines having very interesting properties, especially in the context of the freeness.
In Section \ref{sec: de BE}, we show a de Bruijn-Erd\H{o}s type result which provides a lower bound on the number of intersection points $f_{0}$ for a certain class of conic-line arrangements - this result is directly inspired by combinatorial considerations. In Section \ref{sec: abelian covers}, we provide a Hirzebruch-type inequality for a certain class of conic-line arrangements which leads to a nice bound of Harbourne indices of these arrangements (see Section \ref{sec: local negativity}). In Section \ref{sec: slopes}, we study the geography problem for log-surfaces associated with conic-line arrangements and we provide some extremal examples of those surfaces from the viewpoint of log-Chern slopes.

\paragraph{Notation.} We work over the field of complex numbers.

\section{Chilean arrangement of conics and lines}
\label{sec: Ch arrangement}
   By the way of warm-up we begin with an analysis of a specific conic-line arrangement, which enjoys a number of interesting features,
   not only from the point of view of the present note but more generally. It has been discovered recently
   Dolgachev, Laface, Person, and Urz\'ua during their collaboration in Chile, see \cite{Chilean}.
   For this reason we assign the adjective Chilean to this arrangement. It is also worth mentioning that the same arrangement of conics has been discovered independently by Kohel, Sarti and Roulleau in \cite{SR}, and they apply it in the context of generalized Kummer surfaces. 
   After a brief introductory description, we are going to focus on its freeness in the sense of Saito \cite{Saito}.
\subsection{A general description}
   The famous Hesse arrangement $(12_{3}, 9_{4})$ of $12$ lines and $9$ points with $3$ points on each line
   and $4$ lines through each point, can be derived from the Hesse pencil of cubic curves
$$x^3 + y^3 + z^3 + t\cdot xyz = 0.$$
The pencil has exactly $9$ base points and $4$ reducible members (taking appropriate values of $t$) each consisting of exactly $3$ lines. Blowing up the base points one obtains a rational elliptic surface with four reducible fibers of type $I_{3}$, according to Kodaira's classification of singular fibers. Let us recall that the Hesse pencil of cubic curves can be viewed as the first element in the series of Halphen pencils of plane curves whose general member is of degree $3m$ with nine $m$-multiple points \cite{Halphen}. Here the number $m$ is called the index of the Halphen pencil. It turns out that one can construct explicitly a very interesting arrangement of $12$ conics which is based on Halphen pencil of index $2$. Each of $12$ conics contains $6$ base points and each base point has multiplicity $8$, forming a point-conic configuration $(12_{6}, 9_{8})$ analogously to the Hesse configuration. In fact, the $12$ conics determine a configuration of $9$ points of multiplicity $8$ and $12$ additional points,
where pairs of conics intersect.
   Somehow surprisingly, these $12$ points are exactly the triple intersection points of the dual Hesse $\mathcal{dH}$ arrangement
   $(9_4, 12_3)$ of lines. Thus we have an arrangement of $12$ conics and $9$ lines with the following numerical data for its singular locus
$$t_{9} = 9, \quad t_{5}=12, \quad t_{2}=72.$$
   We follow the authors of \cite{Chilean} and call the arrangement of $12$ conics \emph{Chilean} and denote it by $\mathcal{CH}$.
   The arrangement of $12$ conics and $9$ lines we be called \emph{extended Chilean} and we denote it by $\mathcal{EC}$.
   These three arrangements $\mathcal{H}, \mathcal{CH}$, and $\mathcal{EC}$ enjoy interesting properties from the viewpoint of free divisors.
\subsection{On the freeness of conic-line arrangements}
In this subsection we are going to study the freeness of rational curve arrangements. This study was initiated by Schenck and Toh\v{a}neanu in \cite{ScTo}.
We briefly summarize the basic concepts.
Let $ \mathcal{C} \subseteq \mathbb{P}^{n}$ be an arrangement of reduced and irreducible hypersurfaces and let $\mathcal{C}=V(F)$,  where $F=f_{1}\cdot  ... \cdot f_{d}$ with ${\rm gcd}(f_{i},f_{j}) = 1$.
   Let us denote by ${\rm Der}(S) = S \cdot \partial_{x_{0}} + ... + S \cdot \partial_{x_{n}}$ the ring of polynomial derivations, where $S = \mathbb{C}[x_{0}, ..., x_{n}]$.
   For $\theta \in {\rm Der}(S)$  we have
\begin{equation}\label{eq: Leibnitz}
   \theta(f_{1} \cdot ... \cdot f_{d}) = f_{1} \cdot \theta(f_{2} \cdot ... \cdot f_{d}) + f_{2} \cdot ... \cdot f_{d} \cdot \theta (f_{1})
\end{equation}
   We define the ring of polynomial derivations tangent to the arrangement $\mathcal{C}$ as
$$ {\rm D}(\mathcal{C}) =   \{\theta \in {\rm Der}(S) \,|\, \theta(F) \in \langle F \rangle \}.$$
   Inductive application of formula \eqref{eq: Leibnitz} shows that
$${\rm D}(\mathcal{C}) = \{\theta \in {\rm Der}(S) \, |
\, \theta(f_{i})\in \langle f_{i} \rangle, \,i= 1,...,d \}.$$
   Any arrangement of (reduced) hypersurfaces will have a singular locus of codimension two. Also, in characteristic zero, and as in the case of linear arrangements, we have the following decomposition
$${\rm D}(C)\cong E \oplus {\rm D}_{0}(\mathcal{C}),$$
where $E$ is the Euler derivation and ${\rm D}_{0}(\mathcal{C}) = {\rm syz}(J_{F})$ is the module of syzygies of the Jacobian ideal of the defining polynomial $F$ of $\mathcal{C}$. The freeness of $\mathcal{C}$  boils down to a question whether ${\rm pdim}(S/J_{F}) = 2$, which is equivalent to $J_{F}$ being Cohen-Macaulay. One can show that $\mathcal{C} \subset \mathbb{P}^{2}$ given by $F=0$ is free if the following condition holds: the minimal resolution of the Milnor algebra $M(F) := S / J_{F}$ has the following short form
$$0 \rightarrow S(-d_{1} -(d-1)) \oplus S(-d_{2}-(d-1)) \rightarrow S^{3}(-d+1) \rightarrow S,$$
and the integers $d_{1} \leq d_{2}$ are called the exponents of $\mathcal{C}$.

Finally, let us recall that a singular point $P$ of the arrangement $\mathcal{C}$ is quasi-homogeneous, if its Tjurina number and Milnor number coincide.

Now we are ready to formulate the main results devoted to the freeness of our arrangements.

\begin{proposition}
The arrangement $\mathcal{CH} \subseteq \mathbb{P}^{2}_{\mathbb{C}}$ is free and not all singular points of the arrangement are quasi-homogeneous.
\end{proposition}
\begin{proof}
Our proof is based on computer aided methods with use of \verb}Singular} \cite{Singular}. We can compute the minimal resolution of the Jacobian ideal of $\mathcal{CH}$ which has the following form
$$ 0 \rightarrow S(-7 - 23) \oplus S(-16 - 23) \rightarrow S^{3}(-23) \rightarrow \,S$$
so the exponents of $\mathcal{CH}$ are the integers $d_{1} = 7$, $d_{2} = 16$.
We can compute the total Tjurina number of $\mathcal{CH}$ with use of the exponents (see for instance \cite{Artal}) which is equal to
$$\tau(\mathcal{CH}) = (d-1)^{2} - d_{1}d_{2} = 529 - 112 = 417.$$
On the other hand, the total Milnor number of $\mathcal{CH}$ is equal to
$$\mu(\mathcal{CH}) = \sum_{p \in {\rm Sing}(\mathcal{CH})} (m_{p}-1)^{2} = 9\cdot (8-1)^{2} + 12 \cdot(2-1)^{2} = 441 + 12 = 453,$$
which means that $\tau(\mathcal{CH})  < \mu(\mathcal{CH})$ and this is the reason why all singular points cannot be quasi-homogeneous.
In fact, one can observe that for each $8$-fold point we have $\mu_{p} - \tau_{p} = 4$.
\end{proof}

It is classically known that every reflection arrangement of hyperplanes, i.e., an arrangement which comes from a finite complex reflection group (classified by Shephard and Todd \cite{ST}) is free, and we can easily find the exponents of such an arrangement (see for instance \cite[Theorem 8.3]{Dimca1}). In the case of the dual-Hesse arrangement $\mathcal{H}$ of $9$ lines and $12$ triple intersection points we can compute the Poincar\'e polynomial, namely
$$\chi(\mathcal{H}; t) = 1 + 9t + 24t^{2} + 16t^3 = (1+t)(1+4t)^{2},$$
and the exponents are equal to $d_{1}=d_{2}=4$.

It is natural to ask whether the extended Chilean arrangement of conics and lines $\mathcal{EC}$ is free and, to our surprise, this is actually the case.
\begin{proposition}
The extended Chilean arrangement $\mathcal{EC}$ of $12$ conics and $9$ lines is free.
\end{proposition}
\begin{proof}
Again, our proof is based on computer aided methods with use of \verb}Singular} \cite{Singular}. We compute the minimal free resolution which has the following form
$$ 0 \rightarrow S(-16 - 32) \oplus S(-16 - 32) \rightarrow S^{3}(-32) \rightarrow \,S$$
so the exponents of $\mathcal{EC}$ are the integers $d_{1} = 16$, $d_{2} = 16$.	
\end{proof}
\begin{remark}
In \cite{Valles}, Vall\'es provides an interesting construction of (new) free divisors as the unions of all singular members of a pencil of plane projective curves with the same degree and with a smooth base locus provided that the associated Jacobian ideal is locally a complete intersection. The last condition boils down to the fact that a reduced plane curve $C$ is locally a complete intersection if and only if any singularity of $C$ is quasi homogeneous.
\end{remark}
\section{A de Bruijn-Erd\"os-type inequality for conic-line arrangements}
\label{sec: de BE}
Let us recall that a classical result due to de Bruijn and Erd\H{o}s \cite{deBr} tells us that if $\mathcal{L} \subset \mathbb{P}^{2}_{\mathbb{C}}$ is an arrangement of $d$ lines and $\mathcal{L}$ is not a pencil, then the number of intersection points $f_{0}$ is bounded from below by the number of lines $d$, and the equality holds if and only if $\mathcal{L}$ is a quasi-pencil of lines, i.e., this is an arrangement having exactly one point of multiplicity $d-1$ and exactly $d-1$ double intersection points. On the other hand, \cite[Lemma 4.3]{PRSz} tells us that if we consider an arrangement of $k$ conics with only ordinary singularities in the plane which do not intersect simultaneously at four points, then the number of intersection points $f_{0}$ is bounded from below by $k$, and the bound is sharp. In order to see this, we consider an arrangement $(6_{5},6_{5})$ which consists of $6$ conics and $6$ singular points of multiplicity $5$, and on each conic we have exactly $5$ singular points.

  Here we are going to show that for a certain class of conic-line arrangements in the complex projective plane we can find a bound from below on the number of intersection points. Our proof combines both combinatorial and algebraic methods.
\begin{theorem}[de Bruijn-Erd\H{o}s type statement]\label{thm: de BE}
Let $\mathcal{CL} = \{\ell_{1}, ..., \ell_{d}, C_{1}, ..., C_{k}\} \subset \mathbb{P}^{2}_{\mathbb{C}}$ be an arrangement of $d\geq 2$ lines and $k\geq 2$ conics having only ordinary singularities as the intersections. Furthermore, assume that $t_{k+d} = t_{k+d-1} = t_{k+d-2} = t_{k+d-3} = 0$, then one always has
$$f_{0} = \sum_{r\geq 2}t_{r} \geq k+d.$$
\end{theorem}
\begin{proof}
Our strategy goes as follows. Under the assumptions as above we are going to show that on each line from the arrangement we have at least $2$ singular points from the set of all intersection points ${\rm Sing}(\mathcal{CL})$, and on each conic we have at least $5$ singular points from the set ${\rm Sing}(\mathcal{CL})$. Taking this feature as granted, we consider the blowing-up $f : X \rightarrow \mathbb{P}^{2}_{\mathbb{C}}$ along the set ${\rm Sing}(\mathcal{CL})$.
Denote by $\tilde{C}_{i}$ the strict transform of $C_{i}$  and by $\tilde{\ell}_{j}$ the strict transform of $\ell_{j}$. Since on each line we have at least $2$ singular points and on each conic we have at least $5$ singular points, then the self-intersection numbers of the strict transforms are at most equal to $-1$, and these curves are pairwise disjoint. The Picard number of $X$ is equal to $1 + f_{0}$, and we have produced ad hoc $k+d$ independent $(1,1)$-Hodge classes -- the strict transforms of lines and conics. Then by the Hodge index theorem we see that $f_{0} \geq k+d$, which completes the proof. Now we are going to show that on each line $\ell_{i}$ from the arrangement there are at least $2$ singular points from ${\rm Sing}(\mathcal{CL})$ and on each conic there are at least $5$ singular points from ${\rm Sing}(\mathcal{CL})$. Our proof is going to be almost combinatorial in its nature. If $\ell_{i}$ is a line, then by B\'ezout theorem such a line intersects a conic from the arrangement at exactly $2$ distinct points, which completes our justification for lines. Now let us take a conic $C_{i}$ and assume that there are exactly $4$ points from ${\rm Sing}(\mathcal{CL})$ on it, let use denote these points by $\mathcal{P} = \{P_{1}, P_{2},P_{3}, P_{4}\}$. We need to consider the following two cases:
\begin{itemize}
	\item if $\ell_{i}$ is a line from the arrangement, then it must pass through a pair of two distinct points from $\mathcal{P}$ -- otherwise such a line would intersect $C_{i}$ at an additional point different that $P_{1}, P_{2}, P_{3}, P_{4}$, which would contradict our assumption that we have exactly $4$ singular points on $C_{i}$. Since this property holds for all lines, then this implies that the number of lines is bounded by $\binom{4}{2} = 6$.
	\item A similar argument as above shows that also all conics must pass through the points from $\mathcal{P}$ -- it is obvious for $k=2$, and if $k\geq 3$ we can use the same argument as above -- a conic which does not pass through all points from $\mathcal{P}$ will give an additional intersection points which contradicts the assumption.
\end{itemize}
The resulting arrangement is quite rigid, namely our conic-line arrangement $\mathcal{CL}$ consists of a pencil of conics passing through the points from $\mathcal{P}$ and $d \leq 6$ lines passing through these four points. An easy inspection shows that such arrangements are forbidden due to the fact that $t_{k+d} = t_{k+d-1} = t_{k+d-2} = t_{k+d-3} = 0$, which completes the proof.
\end{proof}
   Now we show that the assumptions in Theorem \ref{thm: de BE} cannot be weakened.
   To this end consider the following arrangement. We take $4$ points in $\mathbb{P}^{2}_{\mathbb{C}}$ and
   all $6$ lines through pairs of these points, i.e. $d=6$, and $3$ conics through
   the points, i.e. $k=3$.
   The following picture visualizes this example.
\begin{figure}[ht]
	\centering
\begin{tikzpicture}[line cap=round,line join=round,>=triangle 45,x=1.0cm,y=1.0cm,scale=2]
\clip(-4.288436071996751,-0.7019534522412448) rectangle (2.722150980780643,2.230032171026375);
\draw [rotate around={88.37366973986529:(-0.7584258227229205,3.630316730738894)},line width=2.pt] (-0.7584258227229205,3.630316730738894) ellipse (4.0967477479719cm and 1.100168247104001cm);
\draw [rotate around={-173.08017328510158:(-1.2823820535985277,0.9099295647523807)},line width=2.pt] (-1.2823820535985277,0.9099295647523807) ellipse (1.4590168654980742cm and 0.6008647161513444cm);
\draw [rotate around={-170.03279009197303:(-1.0060975719443621,0.9832265460752015)},line width=2.pt] (-1.0060975719443621,0.9832265460752015) ellipse (1.1213641209607683cm and 0.7045156166144871cm);
\draw [line width=1.pt,domain=-4.288436071996751:2.722150980780643] plot(\x,{(-1.6226485666017063-1.1451996182862345*\x)/0.2620163962860822});
\draw [line width=1.pt,domain=-4.288436071996751:2.722150980780643] plot(\x,{(--1.024012692246754--0.4121494751435718*\x)/1.4176663485718617});
\draw [line width=1.pt,domain=-4.288436071996751:2.722150980780643] plot(\x,{(--2.305247121504353-0.17824976142889692*\x)/1.8216991411440264});
\draw [line width=1.pt,domain=-4.288436071996751:2.722150980780643] plot(\x,{(-0.13638834677791872-0.5548003817137659*\x)/-0.14201639628608237});
\draw [line width=1.pt,domain=-4.288436071996751:2.722150980780643] plot(\x,{(--1.1328650070951032-0.7330501431426628*\x)/1.6796827448579439});
\draw [line width=1.pt,domain=-4.288436071996751:2.722150980780643] plot(\x,{(--1.8883908871640123--0.9669498568573377*\x)/1.5596827448579442});
\end{tikzpicture}
\caption{A pencil type arrangement of conics and lines.}
\end{figure}
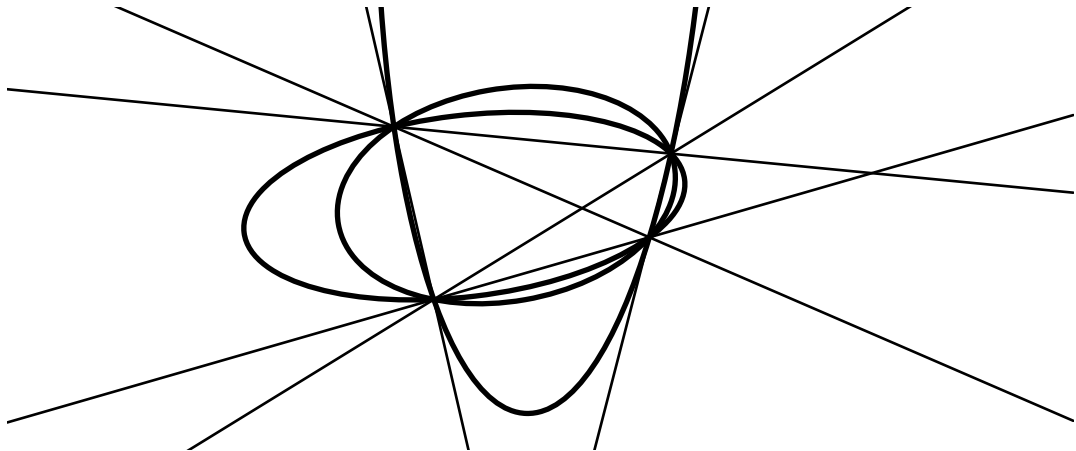
\newpage
\section{Abelian covers of conic-line arrangements}
\label{sec: abelian covers}
In this section we are going to construct an abelian cover branched along conic-line arrangements having only ordinary singularities in order to obtain a Hirzebruch-type inequality. It is worth mentioning in this place that the first author has obtained an inequality in this spirit in \cite{Pokora2} with the use of an orbifold Miyaoka-Yau inequality. However, it turned out that this result is not easily applicable, for instance towards the local negativity phenomenon, or the theory of log-Chern slopes. Due to the mentioned reasons we decided to use a classical construction of abelian covers, which is involving and technical in its nature, but the outcome of this move meets with our expectation (please confront the next section). In the case of conic-line arrangements the existence of a suitable abelian cover is not straightforward \cite{Rita}, and in order to complete our construction we are going to use a general results due to Namba \cite{Namba}.

Let $M$ be a smooth complex projective variety, let $D_{1},\dots,D_{s}$ be irreducible reduced divisors
in $M$ and let $n_{1},\dots,n_{s} \in \mathbb{Z}_{\geq 1}$. Consider the divisor $D =\sum n_{i}D_{i}$. Denote by $Div(M,D)$ the sub-group
of the $\mathbb{Q}$-divisors generated by the entire divisors and
\[
\frac{1}{n_{1}}D_{1},\dots,\frac{1}{n_{s}}D_{s}.
\]
Let $\sim$ be the linear equivalence in $Div(M,D)$, where $G\sim G'$
if and only if $G-G'$ is an entire principal divisor. Let $Div(M,D)/\sim$
be the quotient and let $Div^{0}(M,D)/\sim$ be the kernel of the
Chern class map
\[
\begin{array}{ccc}
Div(M,D)/\sim & \to & H^{1,1}(M,\mathbb{R})\\
G & \to & c_{1}(G)
\end{array}.
\]

\begin{theorem}
	\label{thm:(Namba).-There-exists}(\cite[Theorem 2.3.20]{Namba}).
	There exists a finite abelian cover which branches at $D$ with index
	$n_{i}$ over $D_{i}$ for all $i=1,\dots,s$ if and only if for every
	$j=1,\dots,s$ there exists an element of finite order $v_{j}=\sum\frac{a_{ij}}{n_{i}}D_{i}+E_{j}$
	of $Div^{0}(M,D)/\sim$, where $E_{j}$ is an entire divisor and $a_{ij}\in \mathbb{Z}$,
	such that $a_{jj}$ is co-prime to $n_{j}$. 	Then the subgroup in $Div^{0}(M,D)/\sim$ generated by $v_{j}$'s
	is isomorphic to the Galois group of such an abelian cover.
\end{theorem}

Now we are going to show the main result of the paper.

\begin{theorem}
	\label{hirzconlin1}
Let $\mathcal{CL} = \{\ell_{1}, ..., \ell_{d}, C_{1},..., C_{k}\} \subset \mathbb{P}^{2}_{\mathbb{C}}$ be an arrangement of $d\geq 6$ lines and $k\geq 2$ conics such that all intersection points are ordinary singularities. Moreover, we assume that $t_{d+k}=0$ and one can can find a subarrangement of $\mathcal{CL}$ consisting of $6$ lines intersecting only along double and triple intersection points. Then one has
$$8k + t_{2} + 3t_{3}+t_{4} \geq k + \sum_{r\geq 5}(2r-9)t_{r}.$$
\end{theorem}
\begin{proof}
We are going to adapt the proof of \cite[Theorem A]{PRSz} to our setting. Take $\ell_{i} \in |H|$ and $C_{j} \in |2H|$, where $|H|$ is the linear system of lines in the complex projective plane. We are going to apply Namba's result to $\mathbb{Q}$-divisors $\frac{1}{2}(\ell_{i}-\ell_{j})$ and $\frac{1}{2}C_{j}$. There exists a $(\mathbb{Z}/2\mathbb{Z})^{k+d-1}$ abelian cover $\pi: X \rightarrow \mathbb{P}^{2}_{\mathbb{C}}$ ramified over $\mathcal{CL}$ of order $2$. Let us denote by $\rho: Y \rightarrow X$ the minimal desingularization. Now we are going to follow Hirzebruch's idea to compute the Chern numbers of $Y$. In order to do so, we will use the Fox completion result \cite{Fox}. Let $\tau: Z \rightarrow \mathbb{P}^{2}_{\mathbb{C}}$ be the blowing up of the complex projective plane along all singular points of $\mathcal{CL}$ having multiplicity $m_{p}\geq 3$.
Denote by  $\widetilde{D}=\sum_{i}\widetilde{\ell}_{i} + \sum_{j}\widetilde{C}_{j} $
 the strict transform of $D = \sum_{i} \ell_{i} + \sum_{j} C_{j}$ in $Z$ and let $E_P$ be the exceptional divisor
over the point $P$. There exists a degree $2^{k+d -1}$
map
$$ \sigma:Y \to Z $$
ramified over $Z$ with $\widetilde{D}$ as the branch locus of order $2$, and this gives at the end the following
$$\pi \rho = \tau \sigma.$$

Let $q \in X$ be a singular point, then $C = \rho^{-1}(p)$ is a curve with $p = \pi(q)$ satisfying $m_{p} \geq 3$. Let $E_{p}$ be the exceptional divisor in $Z$ over $p$. Restricting covering $\sigma$ one obtains $\phi: C \rightarrow E_{p}$ of degree $2^{m_{p}-1}$. Using the Hurzwitz formula one gets
$$2-2g(C) = 2^{m_{p}-1}(2-m_{p}) + 2^{m_{p}-2}m_{p} = 2^{m_{p}-2}(4-m_{p}).$$
We want to compute Chern numbers of $Y$. Since this procedure is classically known, let us present only a sketch of the whole procedure. For the topological Euler characteristic of $e(Y)$, we have two ingredients, namely
$$e\bigg(Y \setminus \bigcup_{p \in {\rm Sing}(D),\, m_{p} \geq 3} \sigma^{-1}E_{p} \bigg) = 2^{k+d-1}e(\mathbb{P}^{2}_{\mathbb{C}} \setminus D) + 2^{k+d-2}e(D \setminus {\rm Sing}(D)) + 2^{k+d-3}t_{2},$$
which gives
$$\frac{1}{2^{k+d-3}}\cdot e \bigg(Y \setminus \bigcup_{p \in {\rm Sing}(D),\, m_{p} \geq 3} \sigma^{-1}E_{p} \bigg)= 12-4d - 4k + 2f_{1} - 4f_{0} + t_{2},$$
and
$$e(Y) = e\bigg(Y \setminus \bigcup_{p \in {\rm Sing}(D), \, m_{p} \geq 3} \sigma^{-1}E_{p} \bigg) + \sum_{r\geq 3}2^{k+d-3}(4-r)t_{r}.$$
This finally gives
$$\frac{e(Y)}{2^{k+d-3}} = 12 - 4k - 4d + f_{1} - t_{2}.$$
Now we are going to compute the other Chern number, namely $c_{1}^{2}(Y) = K_{Y}^{2}$. The canonical divisor $K_{Y}$ satisfies $K_{Y} = \sigma^{*}K$ for the divisor $K$ defined as follows:
$$K = \tau^{*}K_{\mathbb{P}^{2}_{\mathbb{C}}} + \sum E_{p} + \frac{1}{2}\bigg( \sum E_{p} + \tau^{*}D - \sum m_{p}E_{p}\bigg) = \sum \frac{3-m_{p}}{2}E_{p} + \frac{1}{2}\tau^{*}D + \tau^{*}K_{\mathbb{P}^{2}_{\mathbb{C}}},$$
where all the summations are taken over all the singular points $P$ of $D$ having multiplicity $m_{p}\geq 3$. Since $K_{Y}^{2} = \deg(\sigma)K^{2}$, thus we obtain that
$$\frac{K_{Y}^{2}}{2^{k+d-3}} = 36 -20k-11d +5f_{1} - 9f_{0} + t_{2}.$$
Now we would like to apply the Bogomolov-Miyaoka-Yau inequality \cite{Wahl} for $Y$, i.e.,
$$K_{Y}^{2} \leq 3 e(Y).$$ In order to do so, we need to prove that the Kodaira dimension of $Y$ is non-negative. It is enough to show that $K_{Y}$ is an effective $\mathbb{Q}$-divisor which in fact means that we need to show that $K$ is an effective $\mathbb{Q}$-divisor. Let us observe that $K$ can be written as
$$K = \bigg(-3 + \frac{1}{2}(d+2k)\bigg) H + \sum \, E_{p} \bigg(1 + \frac{1}{2}(1-m_{p})\bigg),$$
where $H$ denotes the pull-back of $\mathcal{O}_{\mathbb{P}^{2}_{\mathbb{C}}}(1)$ under $\tau$ and the sum goes over all singular points $P$ of $D$ having multiplicity $m_{p}\geq 3$. Our aim is to show that $K$ can be written as a combination of effective divisors with positive rational coefficients. Consider the following presentation
$$K = \tau^{*}\bigg( -\frac{1}{2}(\ell_{1} + ... + \ell_{6})\bigg) +\frac{3}{2} \sum E_{p} + \frac{1}{2} \sum_{i} \tilde{D_{i}},$$
where $\ell_{1}, ..., \ell_{6}$ are the lines from the arrangements intersecting only at double and triple points, and $\tilde{D_{i}}$ denotes the proper transform of either a line $\ell_{i}$ or a conic $C_{i}$ under the blowing up $\tau$. Having those properties in hand we can finally present $K$ as
$$K = \sum_{i}a_{i}\tilde{D_{i}} + \sum_{p}b_{p}E_{p}$$
with $a_{i}$ and $b_{i}$ non-negative, which shows that $K$ is an effective $\mathbb{Q}$-divisor.
Now we are in a position to apply the BMY inequality, namely
$$\frac{K_{Y}^{2}-3e(Y)}{2^{k+d-3}} = 2f_{1} - 9f_{0}+d + 4t_{2} - 8k \leq 0,$$
This finally gives us
\begin{equation}
\label{hirzconlin}
8k + t_{2} + 3t_{3} + t_{4} \geq d + \sum_{r\geq 5}(2r-9)t_{r},
\end{equation}
completing the proof.

\end{proof}
Using a (standard) argument due to Miyaoka and Sakai \cite{M84,Sakai} we can improve our inequality (\ref{hirzconlin}).
\begin{theorem}
	Let $\mathcal{CL} = \{\ell_{1}, ..., \ell_{d}, C_{1},..., C_{k}\} \subset \mathbb{P}^{2}_{\mathbb{C}}$ be an arrangement of $d\geq 6$ lines and $k\geq 2$ conics such that all intersection points are ordinary singularities. Moreover, we assume that  $t_{d+k}=0$ and one can can find a subarrangement of $\mathcal{CL}$ consisting of $6$ lines intersecting only along double and triple points. Then one has
	$$8k + t_{2} + \frac{3}{4}t_{3} \geq d + \sum_{r\geq 5}(2r-9)t_{r}.$$
\end{theorem}
\begin{proof}
It is a modification of the reasoning from \cite[Theorem 2.3]{Pokora1}.
\end{proof}

\section{On the local negativity of conic-line arrangements}
\label{sec: local negativity}
In this section we would like to focus on the so-called local negativity. This subject is motivated by an old folklore conjecture which is the bounded negativity conjecture.
\begin{conjecture}[BNC]
Let $X$ be a smooth complex projective surface, then there exists $b(X) \in \mathbb{Z}$ such that for every reduced curve $C \subset X$ one has $C^{2}\geq -b(X)$.
\end{conjecture}
This conjecture is widely open and we have only some special classes of surfaces for which the BNC holds, for instance rational Coble surfaces, $K3$ surfaces, etc. However, it is not known whether the BNC holds for blow-ups of the complex projective plane along $10$ generic points which shows the difficulty of the problem. In order to approach this conjecture Harbourne introduced the notion of $H$-constants, and shortly afterwards the notion of $H$-indices has been introduced.
\begin{definition}
Let $X$ be a smooth complex projective surface and let $C \subset X$ be a reduced curve having $s\geq 1$ singular points ${\rm Sing}(C) = \{P_{1}, ...,P_{s}\}$, then the $H$-index of $C$ is defined as
$$H(C) = \frac{C^{2} - \sum_{P_{i}}{\rm mult}_{P_{i}}(C)^{2}}{s}.$$
\end{definition}
In other words, the $H$-index of $C$ measures the local negativity on the blow-up of $X$ at ${\rm Sing}(C)$. It is worth mentioning that even in this case it is very difficult to find reasonable lower bounds of the values of $H$-indices. Here we are going to present our lower-bound on $H$-indices in the case of conic-line arrangements having ordinary singularities.
\begin{theorem}
Let $\mathcal{CL} \subset \mathbb{P}^{2}_{\mathbb{C}}$ be a conic-line arrangement satisfying the assumptions of Theorem \ref{hirzconlin1}, then one has
$$H(\mathcal{CL}) \geq -4.5.$$
\end{theorem}
\begin{proof}
Observe that we have
$$H(\mathcal{CL}) = \frac{(2k+d)^{2} - \sum_{r\geq 2} r^{2}t_{r}}{f_{0}} = \frac{4k+d - \sum_{r\geq 2}rt_{r}}{f_{0}}.$$
Now we can use (\ref{hirzconlin}), namely
$$-f_{1} \geq \frac{-8k+4t_{2} + d - 9f_{0}}{2},$$
and we can plug the above inequality to obtain
$$H(\mathcal{CL}) = \frac{4k + d - 4k +2t_{2} + d/2}{f_{0}} - 4.5 \geq -4.5,$$
which completes the proof.

\end{proof}
\section{Log-Chern slopes of open surfaces associated with conic-line arrangements}
\label{sec: slopes}
The main aim of this section is to reopen the so-called geography problem of log surfaces associated with conic-line arrangements. Our presentation here is oriented on combinatorial part of the story and this is the reason why we believe that this problem can be easily accessible both for combinatorialists and geometers.

Let $\mathcal{CL} = \{\ell_{1}, ..., \ell_{d}, C_{1}, ..., C_{k}\} \subseteq \mathbb{P}^{2}_{\mathbb{C}}$ be an arrangement of $k$ conics and $d$ lines having only ordinary singularities. Assume additionally that $t_{k+d} = 0$. Consider the blowing up $f : X \rightarrow \mathbb{P}^{2}_{\mathbb{C}}$ along the set of singular points of $\mathcal{CL}$ having multiplicity $\geq 3$. Denote by $\overline{\mathcal{CL}}$ the reduced total transform of $\mathcal{CL}$. Then the pair $(X,\overline{\mathcal{CL}})$ is a log-surface. We can easily compute the Chern numbers of pair $(X, \overline{\mathcal{CL}})$, namely
$$\overline{c}_{1}^{2}(X,\overline{\mathcal{CL}}) = 9 - 5d - 8k + \sum_{r\geq 2}(3r-4)t_{r};$$
$$\overline{c}_{2}(X,\overline{\mathcal{CL}}) = 3 -2d - 2k + \sum_{r\geq 2}(r-1)t_{r}.$$
Let us recall that by a result due to Miyaoka and Sakai, if the logarithmic Kodaira dimension of pair $(X, \overline{\mathcal{CL}})$ is equal to $2$, i.e., $\overline{\kappa}(X, \overline{\mathcal{CL}})=2$, then we have
	$$\overline{c}_{1}^{2}(X,\overline{\mathcal{CL}}) \leq 3 \overline{c}_{2}(X,\overline{\mathcal{CL}}).$$
If we restrict our attention to the case of lines, i.e., $k=0$, then by the result due to Sommese \cite[Theorem 5.1]{Sommese} we know that one always has
$$\overline{c}_{1}^{2}(X,\overline{\mathcal{L}}) \leq \frac{8}{3}\overline{c}_{2}(X,\overline{\mathcal{L}}),$$
and the equality holds if and only if  $\mathcal{L}$ is projectively equivalent to the dual Hesse arrangement of lines $\mathcal{H}$. \\
If the ground field $\mathbb{F}$ is arbitrary, then Eterovi\'c, Figueroa, and Urz\'ua in \cite{EFU} proved that
$$\frac{2d-6}{d-2} \leq \frac{\overline{c}_{1}^{2}(X,\overline{\mathcal{L}})}{\overline{c}_{2}(X,\overline{\mathcal{L}})} \leq 3,$$
and the left-hand side equality holds if and only if $\mathcal{L}$ is a star configuration, i.e., it has only double intersection points, and the right hand side equality holds if and only if $\sum_{r\geq 2}t_{r} = d$, so this is the case, for instance, when $\mathcal{L}$ is a finite projective plane arrangement. Among others, they also found an interesting link between $H$-indices of line arrangements and the limit points of ratios of log-Chern numbers -- in fact one can show that the accumulation points of $H$-indices of line arrangements $\mathcal{L}$ are in one-to-one correspondence with the limit points of $\frac{\overline{c}_{1}^{2}(X,\overline{\mathcal{L}})}{\overline{c}_{2}(X,\overline{\mathcal{L}})}$, please consult \cite[Proposition 4.9]{EFU}.

Now if we focus on the case of conic arrangements, i.e., $d=0$, then the first author in \cite{Pokora1} proved that one always has
$$\overline{c}_{1}^{2}(X,\overline{\mathcal{C}}) < \frac{8}{3} \overline{c}_{2}(X,\overline{\mathcal{C}}),$$
but we do not have any example of a conic arrangement $\mathcal{C}$ such that for the associated pair $(X,\overline{\mathcal{C}})$ one has
$\frac{\overline{c}_{1}^{2}(X,\overline{\mathcal{C}})}{\overline{c}_{2}(X,\overline{\mathcal{C}})} \approx \frac{8}{3}$.

Now we are going to study some extremal conic-line arrangements from the viewpoint of log-Chern slopes. Let us call the ratio $$E(X, \overline{\mathcal{CL}}) := \frac{\overline{c}_{1}^{2}(X,\overline{\mathcal{CL}})}{\overline{c}_{2}(X,\overline{\mathcal{CL}})}$$ the log-Chern slope of $(X, \overline{\mathcal{CL}})$. Before we start our investigations, we want to recall the following (modified) question by Urz\'ua \cite[Question VII.12]{Urzua} which is a main motivation for our studies.

\begin{question}
Let $\mathcal{CL} \subset \mathbb{P}^{2}_{\mathbb{C}}$ be an arrangement of $k\geq 2$ conics and $d\geq 2$ lines such that all intersection points are ordinary. Is it true that the following inequality holds
$$8k + 2t_{2} + t_{3} \geq d+3 + \sum_{r\geq 5}(r-4)t_{r}?$$
In particular, is it true that $E(X, \overline{\mathcal{CL}}) \leq \frac{8}{3}$.
\end{question}
At this moment we are unable to answer this question, but we hope that our work will shed some light on this topic. We strongly believe that the answer to the question above is \textbf{YES}.

There are two steps towards an affirmative answer on that question, namely one needs to show that
\begin{itemize}
	\item $\overline{c}_{2}(X,\overline{\mathcal{CL}}) > 0$, and
	\item there exists a $(\mathbb{Z}/3\mathbb{Z})^{k+d-1}$ abelian cover $\pi: X \rightarrow \mathbb{P}^{2}_{\mathbb{C}}$ ramified along $\mathcal{CL}$ of order $3$.
\end{itemize}
In this context, we are able to show the following result.
\begin{proposition}
Let $\mathcal{CL} \subset \mathbb{P}^{2}_{\mathbb{C}}$ be a conic-line arrangement with ordinary singularities such that $k\geq 2$ and $d\geq 2$ with $t_{k+d}=t_{k+d-1}=0$. Then $\overline{c}_{2}(X,\overline{\mathcal{CL}}) > 0$.

\end{proposition}
\begin{proof}
	We will argue as in \cite[Proposition 3.3]{EFU}. First of all, if we consider the case with $k=d=2$, then the only intersection points are double points, so we have $\overline{c}_{2}(X,\overline{\mathcal{CL}}) = 8$, and we proceed using an induction argument with respect to the number of curves in $\mathcal{CL}$. Suppose that $k+d \geq 5$, we will have two subcases, namely with respect to a line $L$ which is passing through $t\geq 2$ points, or with respect to a conic $C$ which is passing through $r \geq 4$ points. Then
	$$\overline{c}_{2}(X,\overline{\mathcal{CL}}) \geq \overline{c}_{2}(X,\overline{\mathcal{CL} \setminus L }) - 2 + t \geq \overline{c}_{2}(X,\overline{\mathcal{CL} \setminus L }) > 0,$$
	$$\overline{c}_{2}(X,\overline{\mathcal{CL}}) \geq \overline{c}_{2}(X,\overline{\mathcal{CL} \setminus C }) -2 + r \geq \overline{c}_{2}(X,\overline{\mathcal{CL} \setminus C }) + 2 > 0,$$
	which completes the proof.
\end{proof}

Now we are going to study some log-surfaces constructed with use of conic-line arrangements.
\begin{example}[Klein's arrangement of conics and lines]
In \cite{PR2019}, the authors presented in detail Klein's arrangement of conics and lines in the complex projective plane. It consists of $21$ lines and $21$ conics (these curves are polars to Klein's quartic curve at the $21$ quadruple points of Klein's arrangement of $21$ lines), and it has $42$ double points, $252$ triple points, and $189$ quadruple points. Simple computations give
$$E(X,\overline{\mathcal{CL}}) = \frac{\overline{c}_{1}^{2}(X,\overline{\mathcal{CL}})}{\overline{c}_{2}(X,\overline{\mathcal{CL}})} = \frac{9-8\cdot 21 - 5\cdot 21 + 2\cdot 42 + 5\cdot 252 + 8\cdot 189}{3-2\cdot 21 - 2\cdot 21 + 42 + 2\cdot 252 + 3\cdot 189} \approx 2.512,$$ and this is the highest known value of  log-Chern slopes in the class of conic-line arrangements (according to our best knowledge).	
\end{example}
It is natural to ask about the lowest possible value of $E(X,\overline{\mathcal{CL}})$, and here is our candidate.
\begin{example}[Pencil-type arrangement of conics and lines]
Let us consider the case of $d=2$ lines and $k\geq 1$ conics having the following combinatorics
$$t_{k+1} = 4, \quad t_{2} = 1.$$
We can compute the log-Chern numbers, namely
$$\overline{c}_{1}^{2}(X, \overline{\mathcal{CL}}) = 9 - 10 - 8k + 2 + 4(3k-1) = 4k-3;$$
$$\overline{c}_{2}(X, \overline{\mathcal{CL}}) = 3-4-2k + 1 + 4k = 2k,$$
and finally we get
$$E(X,\overline{\mathcal{CL}}) = 2 - \frac{3}{2k}.$$
In particular, if we take $k=1$, we get $E(X,\overline{\mathcal{CL}})  = \frac{1}{2}$, which is the lowest-known value of the log-Chern slope (to our best knowledge) in the class of conic-line arrangements.
\end{example}
\begin{example}[Pappus conic-line arrangement]
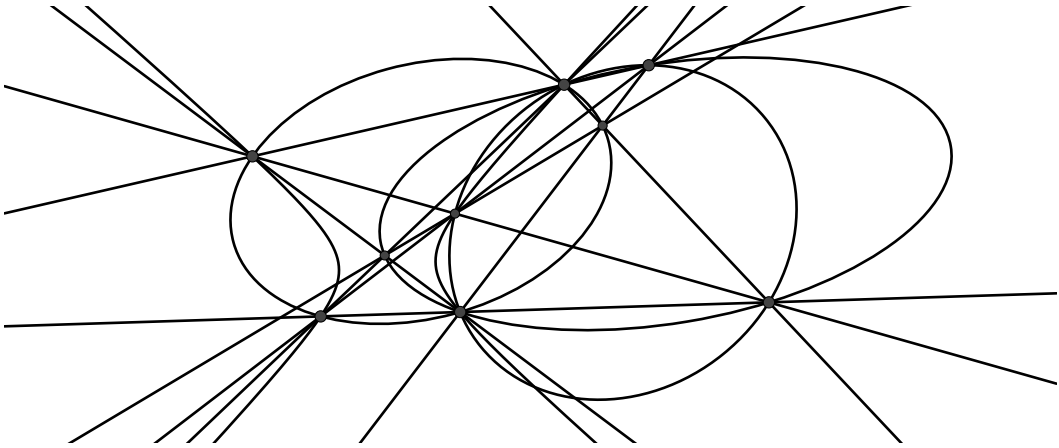
\begin{figure}[h]
\centering
\definecolor{uuuuuu}{rgb}{0.26666666666666666,0.26666666666666666,0.26666666666666666}
\definecolor{ttqqqq}{rgb}{0.2,0.,0.}
\begin{tikzpicture}[line cap=round,line join=round,>=triangle 45,x=1.0cm,y=1.0cm,scale=0.85]
\clip(-8.835774951042856,-2.972286057187465) rectangle (7.446379849886434,3.8372783106395665);
\draw [line width=1.pt,domain=-8.835774951042856:7.446379849886434] plot(\x,{(--16.26805626440505--1.4212194476643834*\x)/6.131546759923489});
\draw [line width=1.pt,domain=-8.835774951042856:7.446379849886434] plot(\x,{(-6.084418791872848--0.2193937300667319*\x)/6.934832889971526});
\draw [line width=1.pt,domain=-8.835774951042856:7.446379849886434] plot(\x,{(-7.3283161855381245-2.4290070285667693*\x)/3.2139855386849687});
\draw [line width=1.pt,domain=-8.835774951042856:7.446379849886434] plot(\x,{(--0.5689281951084708-2.277891870829256*\x)/7.990595908236498});
\draw [line width=1.pt,domain=-8.835774951042856:7.446379849886434] plot(\x,{(-10.46820276416167-3.6150185599129854*\x)/-3.766456319208361});
\draw [line width=1.pt,domain=-8.835774951042856:7.446379849886434] plot(\x,{(--7.6905911501030175-3.3956248298462532*\x)/3.1683765707631655});
\draw [line width=1.pt,domain=-8.835774951042856:7.446379849886434] plot(\x,{(-4.135023522012051-3.850226476231153*\x)/-2.91756122123852});
\draw [line width=1.pt,domain=-8.835774951042856:7.446379849886434] plot(\x,{(-10.351901016352794-3.9185050485603714*\x)/-5.075783741658517});
\draw [line width=1.pt,domain=-8.835774951042856:7.446379849886434] plot(\x,{(--5.801930000389483--2.025919206459573*\x)/3.3715184309258346});
\draw [rotate around={-164.59993786852235:(-2.3898619256226312,0.9485403578049165)},line width=1.pt] (-2.3898619256226312,0.9485403578049165) ellipse (3.0064274765872723cm and 1.9770609808067632cm);
\draw [rotate around={-170.46398746361064:(1.3967703916113146,0.9108135467273735)},line width=1.pt] (1.3967703916113146,0.9108135467273735) ellipse (4.474021527882336cm and 2.0214451912717792cm);
\draw [rotate around={-140.80575584306933:(0.7372780245482458,0.3074158638558599)},line width=1.pt] (0.7372780245482458,0.3074158638558599) ellipse (2.8315953783872327cm and 2.4501059962716805cm);
\draw [samples=50,domain=-0.99:0.99,rotate around={2.7334288256786654:(-2.9095451158676684,-0.21300253594885807)},xshift=-2.9095451158676684cm,yshift=-0.21300253594885807cm,line width=1.pt] plot ({0.7485285378290547*(1+(\x)^2)/(1-(\x)^2)},{0.727924013714328*2*(\x)/(1-(\x)^2)});
\draw [samples=50,domain=-0.99:0.99,rotate around={2.7334288256786654:(-2.9095451158676684,-0.21300253594885807)},xshift=-2.9095451158676684cm,yshift=-0.21300253594885807cm,line width=1.pt] plot ({0.7485285378290547*(-1-(\x)^2)/(1-(\x)^2)},{0.727924013714328*(-2)*(\x)/(1-(\x)^2)});
\begin{scriptsize}
\draw [fill=uuuuuu] (-4.995502177133899,1.4952754628469946) circle (2.5pt);
\draw [fill=uuuuuu] (1.13604458278959,2.916494910511378) circle (2.5pt);
\draw [fill=uuuuuu] (-3.9397391588689272,-1.0020101380489936) circle (2.5pt);
\draw [fill=uuuuuu] (2.995093731102599,-0.7826164079822617) circle (2.5pt);
\draw [fill=uuuuuu] (-0.1732828396605665,2.6130084218639915) circle (2.5pt);
\draw [fill=uuuuuu] (-1.78151663844893,-0.9337315657197748) circle (2.5pt);
\draw [fill=uuuuuu] (-2.9491671876613412,-0.05126605465612959) circle (2.0pt);
\draw [fill=uuuuuu] (-1.8620013942479892,0.6020034162383691) circle (2.0pt);
\draw [fill=uuuuuu] (0.42235124326449325,1.9746531518034434) circle (2.0pt);
\end{scriptsize}
\end{tikzpicture}
\caption{Pappus conic-line arrangement.}
\end{figure}
Using the classical Pappus arrangement $(9_{3},9_{3}))$ we can produce a new arrangement of conics and lines -- as we can observe there are plenty possibilities. The construction is based on the fact that we can choose $5$ sufficiently general points from $9$ marked points in Pappus arrangements such that the resulting conics are irreducible and the intersection points of the new arrangement of $4$ conics and $9$ lines are ordinary. As we can see, we have the following intersection points:
$$t_{7} = 2, \quad t_{5} = 4, \quad t_{4} = 2, \quad t_{2} = 36.$$
Simple computations give
$$E(X,\overline{\mathcal{CL}}) =\frac{\overline{c}_{1}^{2}(X,\overline{\mathcal{CL}})}{\overline{c}_{2}(X,\overline{\mathcal{CL}})} = \frac{9-8\cdot 4 - 5\cdot 9 + 2\cdot 36 + 8\cdot 2 + 11\cdot 4 + 17\cdot 2}{3-2\cdot 4 - 2\cdot 9 + 36 + 3\cdot 2 + 4\cdot 4 + 6\cdot 2} \approx 2.085.$$
\end{example}
\begin{example}[Hesse arrangement of conics and lines]
Finally, we would like to look at the conic-line Hesse arrangement from a viewpoint of log-Chern slopes. Let us recall that the arrangement consists of $d=9$ lines, $k=12$ conics, and
$$t_{2} = 72, \quad t_{5} = 12, \quad t_{9}=9.$$
Then
$$E(X,\overline{\mathcal{EC}}) = \frac{9-5\cdot9 - 8\cdot 12 +2\cdot 72 + 11 \cdot 12 + 23 \cdot 9}{3 - 2\cdot 12 - 2\cdot 9 + 72 + 4\cdot 12 + 8\cdot 9}  \approx 2.294.$$
\end{example}
At the end of our investigations, we would like to ask a general structural question.
\begin{question}
Let $\mathcal{CL} \subseteq \mathbb{P}^{2}_{\mathbb{R}}$ be a conic-line arrangement having only ordinary singularities. Is it true that
$$E(X, \overline{\mathcal{CL}}) \leq 5/2 \,?$$
\end{question}
This question strictly depends on the problem whether we can construct highly singular examples of conic-line arrangements over the reals, and this is an interesting combinatorial problem. However, we do not know how to apporoach this problem efficiently.

At last, we would like to present an example of a specific family of log-surfaces, inspired by the above investigations, with log-Chern slopes $\leq 1$.
\begin{example}
If $\mathcal{C} = \{C_{1}, ..., C_{k}\} \subset\mathbb{P}^{2}_{\mathbb{C}}$ such that ${\rm deg}(C_{i}) = d_{i} \geq 1$ and all intersection points are ordinary. As usually, we assume that $t_{k}=0$. The Chern numbers for such a pair $(X,\overline{\mathcal{C}})$ are equal to:
$$\bar{c}_{1}^{2}(X,\overline{\mathcal{C}}) = 9+\sum_{i=1}^{k}(d_{i}^{2}-6d_{i})+3f_{1} - 4f_{0},$$
$$\bar{c}_{2}(X, \overline{\mathcal{C}}) = 3 + \sum_{i=1}^{k}(d^{2}_{i}-3d_{i}) + f_{1}-f_{0}.$$
Consider $\mathcal{C}_{d} = \{\ell_{1}, \ell_{2}, C_{3}\}$, where $\ell_{i}$'s are lines and $C_{3}$ is a smooth curve of degree $d\geq 2$, and all intersection points are the ordinary double points with $t_{2} = 2d+1$. Then we have
$$E(\mathbb{P}^{2}_{\mathbb{C}},\mathcal{C}_{d}) = \frac{(d-1)^{2}}{d(d-1)} = \frac{d-1}{d},$$
so we obtain a sequence of log-Chern slopes (depending on $d$) starting with $E(\mathbb{P}^{2}_{\mathbb{C}},\mathcal{C}_{2}) = \frac{1}{2}$ and  $E(\mathbb{P}^{2}_{\mathbb{C}},\mathcal{C}_{d}) \rightarrow 1$ provided that $d \rightarrow \infty$.
Let us denote by $C = \ell_{1} + \ell_{2} + C_{3}$ the associated divisor and our log surfaces by $Y = \mathbb{P}^{2}_{\mathbb{C}} \setminus C$. Using a result due to Fulton and Zariski \cite{Fulton} we know that if $C$ is a curve in $\mathbb{P}^{2}_{\mathbb{C}}$ whose only singularities are ordinary double points with distinct tangents, then any covering of $\mathbb{P}^{2}_{\mathbb{C}}$ branched along $C$ is abelian. This means that we have the following isomorphism
$$\pi_{1}(\mathbb{P}^{2}_{\mathbb{C}} \setminus C) = (\mathbb{Z} \oplus \mathbb{Z} \oplus \mathbb{Z})/(1,1,d)$$
and it provides an explicit description of the fundamental group of $Y$.
\end{example}
\paragraph*{Acknowledgement.}
We would like to warmly thank Xavier Roulleau for useful comments, especially for suggesting to use Namba covers, and to Giancarlo Urz\'ua for useful discussions around the Chilean arrangement of conics. We would like to thank Jose Ignacio Cogolludo and Anatoly Libgober for interesting remarks about the content of our paper and for pointing out \cite{Valles}. Finally, we want to thank anonymous referees for their critical comments that allowed us to improve the present paper.
The first author was partially supported by National Science Center (Poland) Sonata Grant \textbf{Nr 2018/31/D/ST1/00177}, and the second author was partially supported by National Science Centre (Poland) Harmonia Grant \textbf{Nr 2018/30/M/ST1/00148}.
\paragraph*{Data availability.} Not applicable as the results presented in this manuscript rely on no external sources of
data or code.



\bibliographystyle{abbrv}
\bibliography{master}



\bigskip
\bigskip
\noindent
Piotr Pokora,\\
Department of Mathematics, Pedagogical University of Cracow,
Podchor\c a\.zych 2,
PL-30-084 Krak\'ow, Poland.

\nopagebreak
\noindent
\textit{E-mail address:} \texttt{piotr.pokora@up.krakow.pl, piotrpkr@gmail.com}\\

\bigskip
\noindent
   Tomasz Szemberg,\\
   Department of Mathematics, Pedagogical University of Cracow,
   Podchor\c a\.zych 2,
   PL-30-084 Krak\'ow, Poland.

\nopagebreak
\noindent
   \textit{E-mail address:} \texttt{tomasz.szemberg@gmail.com}\\

\end{document}